\documentclass{amsart}
\usepackage{amsmath, amsthm}
\usepackage{amsfonts}
\usepackage{amssymb}
\usepackage{amsthm}
\usepackage{amsmath}
\usepackage{latexsym}
\usepackage[all]{xy}
\usepackage{mathrsfs}
\usepackage{graphicx}

\newtheoremstyle{theorem}%name
  {10pt}		  % space above
  {10pt}  % space below
  {\sl}  % bofy font
  {\parindent}     % ident - empty=no indent,  \parindent= paragraph indent
  {\bf}  % thm head font
  {. }    % punctuation after thm head
  { }    % space after thm head: `` ``=normal \newline=linebreak
  {}     % thm head specification
\theoremstyle{theorem}
\newtheorem{theorem}{Theorem}

\newtheorem{lemma}[theorem]{Lemma}

\newtheoremstyle{defi}%name
  {10pt}		  % space above
  {10pt}  % space below
  {\rm}  % bofy font
  {\parindent}     % ident - empty=no indent,  \parindent= paragraph indent
  {\bf}  % thm head font
  {. }    % punctuation after thm head
  { }    % space after thm head: `` ``=normal \newline=linebreak
  {}     % thm head specification
\theoremstyle{defi}

\title{Periodicity of complementing multisets}
\author{\v Zeljka Ljuji\' c}
\address{Mathematics Ph.D. Program\\
The CUNY Graduate Center\\
365 Fifth Avenue\\
New York NY 10016}
\email{ zljujic@gc.cuny.edu}
\date{}

\begin{document}

\maketitle

\begin{abstract}
Let $A$ be a finite multiset of integers. If $B$ be a multiset such that $A$ and $B$ are $t$-complementing multisets of integers, then $B$ is periodic. We obtain the Biro-type upper bound for the smallest such period of $B$: Let $\varepsilon>0$. We assume that $\textrm{diam}(A)\ge n_0(\varepsilon)$ and that $\sum_{a\in A}w_A(a)\leq (\textrm{diam}(A)+1)^{c}$, where $c$ is any constant such that $c< 100\log2-2$. Then $B$ is periodic with period 
\[\log k\leq (\textrm{diam}(A)+1)^{\frac{1}{3}+\varepsilon}.
\]
\end{abstract}

\section{Introduction}
 
Let $A$ and $B$ be subsets of integers. The sumset $A+B$ 
is the set of all integers of the form $a+b$, where $a\in A$ and $b\in B$. If every integer 
has a unique representation as the sum of an element of $A$ and 
an element of $B$, then we write $A 
\oplus B = \mathbb{Z}$ and we say that $A$ and $B$ are \emph{complementing sets of integers}. 

Let $A$ be a finite set of integers. One of the classical problems is to decide
whether there exists an infinite set $B$ such that $A\oplus B = \mathbb{Z}$. 

We say that a set $B\subset\mathbb{Z}$ is \emph{periodic} if there exists $k\in\mathbb{Z}_{>0}$ such that $B+\{k\}=B$. In that case, we say that $k$ is a \emph{period} of $B$. An early result of D.J. Newman \cite{a} states the following:

\begin{theorem}[D.J. Newman \cite{a}] Let $A$ be a nonempty finite set of integers and let $\textrm{diam}(A)=\max{(A)}-\min{(A)}$. If there exists a set $B$ such that $A\oplus B=\mathbb{Z}$, then $B$ is periodic with period 
\[k \le 2^{\textrm{diam}(A)} .\]
 
\end{theorem}

From here, one can ask a natural question: What is the best upper bound for the period in terms of $\textrm{diam}(A)$? I.Z. Ruzsa, [in Tijdeman \cite{d}, Appendix], translated the problem into a problem of divisibility of certain integer polynomials and proved 
\[\log k <<\sqrt{\textrm{diam}(A)\log (\textrm{diam}(A)}).\]
M. Koloutzakis \cite{b},using the same method, obtained a slightly weaker bound. A. Biro \cite{e}, improved the Ruzsa's result on the divisibility of integer polynomials and obtained the following bound.

\begin{theorem}[A. Biro \cite{e}] For every $\varepsilon>0$ there exists an integer $n_0$ with the following property: Let $A$ and $B$ be sets of integers such that $A$ is finite and $A \oplus B = \mathbb{Z}$. Then, if $\textrm{diam}(A)\ge n_0$, there exists a period $k$ of $B$ such that
\[\log k\le \textrm{diam}(A)^{\frac{1}{3}+\epsilon}.\] 

\end{theorem}

The problem of complementing sets of integers was generalized to linear forms by M.B. Nathanson \cite{f} as follows: we consider two linear forms
\[\psi(x_1,\ldots,x_h)=u_1x_1+\cdots+u_hx_h
\] 
and
\[\rho(x_1,\ldots,x_h,y)=\psi(x_1,\ldots,x_h)+vy,
\]
with nonzero integer coefficients  $u_1,\ldots,u_h,v$. Let $\mathcal{A}=(A_1,\ldots,A_h)$ be an $h$-tuple of nonempty finite sets of integers and $B$ a set of integers. We introduce the sets 
\[\psi(\mathcal{A})=\{u_1a_1+\cdots+u_ha_h: a_i\in A_i\}
\]
and
\[\rho(\mathcal{A},B)=\{u_1a_1+\cdots+u_ha_h+vb: a_i\in A_i, b\in B\}.
\]
We denote $\textrm{diam}(\psi(\mathcal{A}))=\max(\psi(\mathcal{A}))-\min(\psi(\mathcal{A}))$.

\noindent For every integer $n$, we define the representation function associated to $\psi$
\[R_{\mathcal{A}}^{(\psi)}(n)=\textrm{card}(\{(a_1,\ldots,a_h)\in A_1\times\cdots\times A_h:\psi(a_1,\ldots,a_h,)=n\}),
\]
and the representation function associated to $\rho$ by
\[R_{\mathcal{A},B}^{(\rho)}(n)=\textrm{card}(\{(a_1,\ldots,a_h,b)\in A_1\times\cdots\times A_h\times B:\rho(a_1,\ldots,a_h,b)=n\}).
\]

\noindent We say that $\mathcal{A}$ and $B$ are \emph{complementing sets of integers with respect to the linear form $\rho$} if $\rho(\mathcal{A},B)=\mathbb{Z}$ and $R_{\mathcal{A},B}^{(\rho)}(n)=1$, for all integers $n$. Similarly, $\mathcal{A}$ and $B$ are \emph{$t$-complementing sets of integers with respect to $\rho$} if $\rho(\mathcal{A},B)=\mathbb{Z}$ and $R_{\mathcal{A},B}^{(\rho)}(n)=t$, for all integers $n$.

\begin{theorem} [M.B. Nathanson \cite{f}] Let $h\geq 1$ and let
\[\rho(x_1,\ldots,x_h,y)=\psi(x_1,\ldots,x_h)+vy
\]
be a linear form with nonzero integer coefficients $u_1,\ldots,u_h, v$. Let $\mathcal{A}=(A_1,\ldots,A_h)$ be an $h$-tuple of nonempty finite sets of integers. If $\mathcal{A}$ and $B$ are $t$-complementing sets of integers with respect to $\rho$, then $B$ is periodic with period 
\[k \le 2^\frac{\textrm{diam}(\psi(A))}{|v|} .\]

\end{theorem}

Here, we consider the complementing multisets of integers. More precisely, let $S$ be a multiset of integers. For each integer $n$, we denote by $w_S(n)\in\mathbb{Z}_{\ge0}$ the weight of $n$ in $S$, which is the number of occurrences of $n$ in $S$. Let $A$ be a finite multiset of integers and let $B$ be a multiset of integers. For every integer n, we define the representation function associated to the multisets $A$ and $B$ as
\[R_{A,B}(n)=\sum_{\substack{
                                    n=a+b\\
                                    a\in A, b\in B}}w_A(a)w_B(b).
\]
Let $t\in\mathbb{Z}_{>0}$. We say that $A$ and $B$ are \emph{$t$-complementing multisets of integers} if $A+B=\mathbb{Z}$ and $R_{A,B}(n)=t$, for all $n\in\mathbb{Z}$. In that case, we write $A\oplus_t B=\mathbb{Z}$.

We say that a multiset $B$ is \emph{periodic} if there exists $k\in\mathbb{Z}_{>0}$ such that $w_B(n+k)=w_B(n)$, for all $n\in \mathbb{Z}$. Any such $k$ is called a \emph{period} of a multiset $B$. More generally, we say that a multiset of integers $B$ is \emph{eventually periodic} if there exist $k\in\mathbb{Z}_{>0}$ and $n_0\in\mathbb{Z}$ such that if $n\ge n_0$, then $w_B(n+k)=w_B(n)$. In that case, $k$ is an \emph{eventual period} of $B$. 

Similarly, a representation function $R_{A,B}$ is \emph{eventually periodic} if there exist $m\in\mathbb{Z}_{>0}$ and $n_0\in\mathbb{Z}$ such that if $n\ge n_0$ we have $R_{A,B}(n+m)=R_{A,B}(n)$. An integer $m$ is called an \emph{eventual period} of $R_{A,B}$.

Note that in the case $t=1$ the multisets $A$ and $B$ are an ordinary sets and the problem of complementing multisets becomes the classical problem of complementing sets of integers. In the case of complementing sets with respect to linear forms, we can consider $\psi(\mathcal{A})$ as a multiset $A'$. More precisely, if $r=u_1a_1+\cdots+u_ha_h\in \psi(\mathcal{A})$, we define $w_{\psi(\mathcal{A})}(r)=R_{\mathcal{A}}^{(\psi)}(r)$. Then, if there exists a set $B$ such that $\mathcal{A}$ and $B$ are $t$-complementing sets of integers with respect to $\rho$, we have that $A'\oplus_t B'=\mathbb{Z}$, where $B'=vB$ and $B$ is periodic with period $k$ if and only if $B'$ is periodic with period $vk$.

In this paper we prove the following equivalent of Theorem 1 [D.J. Newman \cite{a}] in the case of multisets:

\begin{theorem} Let $A$ be a nonempty finite multiset of integers. Let $\textrm{diam}(A)=\max{(A)}-\min{(A)}$. If there exists a multiset $B$ such that $A\oplus_t B=\mathbb{Z}$, then $B$ is periodic with period 
\[k \le (t+1)^{\textrm{diam}(A)} .\]

\end{theorem}

Moreover, we follow Ruzsa's idea and translate the problem into the problem of the divisibility of integer polynomials. We extend the main theorem in \cite{e} to fit our purpose and  we obtain the following theorem.

\begin{theorem} For every $\varepsilon>0$ there exists an integer $n_0$ with the following property: Let $A$ be a finite multiset of integers such that $|A|>1$. Suppose that $B$ is an eventually periodic infinite multiset of integers with eventual period $k$, and that the representation function $R_{A,B}$ is eventually periodic with eventual period $m$. If $n=\textrm{diam}(A)+m\geq n_0$ and $\sum_{a\in A}w_A(a)\leq n^{c}$, where $c< 100\log2-2$, then there exists an eventual period $k$ of $B$ such that
\[\log k\leq n^{\frac{1}{3}+\varepsilon}.
\]
\end{theorem}

As an immediate corollary, we obtain a new upper bound of the period of $t$-complementing multisets of integers. 

\begin{theorem} For every $\varepsilon>0$ there exists an integer $n_0$ with the following property: Let $A$ be a nonempty finite multiset of integers and let $\textrm{diam}(A)=\max{(A)}-\min{(A)}$. We assume that $\textrm{diam}(A)\ge n_0$ and that $\sum_{a\in A}w_A(a)\leq (\textrm{diam}(A)+1)^{c}$, where $c< 100\log2-2$. If $B$ is a multiset such that $A\oplus_t B=\mathbb{Z}$, then $B$ is periodic with period 
\[\log k\leq (\textrm{diam}(A)+1)^{\frac{1}{3}+\varepsilon}.
\]
\end{theorem}

The last theorem can be restated in terms of complementing sets of integers with respect to linear forms.

\begin{theorem} For every $\varepsilon>0$ there exists an integer $n_0$ with the following property: Let $h\geq 1$ and let
\[\rho(x_1,\ldots,x_h,y)=\psi(x_1,\ldots,x_h)+vy
\]
be a linear form with nonzero integer coefficients $u_1,\ldots,u_h, v$. Let $\mathcal{A}=(A_1,\ldots,A_h)$ be an $h$-tuple of nonempty finite sets of integers and let $\textrm{diam}(\psi(A))=\max(\psi(\mathcal{A}))-\min(\psi(\mathcal{A}))$. We assume that $\textrm{diam}(\psi(A))\geq n_0$ and that  $\prod_{i=1}^h|A_i|\leq n^c$, where $c< 100\log2-2$. If $B$ is a set such that $\mathcal{A}$ and $B$ are $t$-complementing sets of integers with respect to $\rho$, then $B$ is periodic with period 
\[\log k\leq (\textrm{diam}(\psi(A))+1)^{\frac{1}{3}+\varepsilon}.
\]
\end{theorem}

\section{Preliminaries}

We start by introducing some notation. Let $n\in\mathbb{Z}_{>0}$. 

If $p$ is a prime number such that $p^r|n$, but $p^{r+1}\nmid n$, for some $r\in\mathbb{Z}_{\ge 1}$, we write $p^r\parallel n$. 

The set of all primitive $n$-th roots of unity will be denoted by $\mu_n$.  Then $\phi(x)=|\mu_n|$ denotes Euler's function and $\Phi_n(x)=\prod_{\xi\in \mu_n}(x-\xi)$ denotes the $n$-th cyclotomic polynomial. The number of divisors of n will be denoted by $\tau(n)$ and the number of distinct prime divisors of n by $\omega(n)$. We denote the M\"{o}bius function by $\mu(n)$.

If $f(x)=\sum a_ix^i\in\mathbb{Z}[x]$,  then $\|f(x)\|=\sum |a_i|$. It is easy to see that if $f_1,f_2\in\mathbb{Z}[x]$, then $\|f_1(x)f_2(x)\|\le \|f_1(x)\|\|f_2(x)\|$. 

As usual, $\log$ will denote the natural logarithm.

The following three lemmas are Lemma $1$, Lemma $2$ and Lemma $3$ from \cite{e}. 

\begin{lemma}\label{l1} Let $f(x)\in \mathbb{C}[x]$ be a nonzero polynomial such that $x^d-1|f(x)$, for some $d\in\mathbb{Z}_{>0}$. Let $n=\deg f(x)$ let $g(x)=\frac{f(x)}{x^d-1}$.  Then
\[\|g(x)\| \le n\|f(x)\|.
\] 
\end{lemma}

\begin{lemma}\label{l2} Let $\varepsilon>0$ and let $f(x)\in\mathbb{C}[x]$ be a nonzero polynomial such that $\deg f(x)\leq n$, where $n\geq n_0(\varepsilon)$. Let $m$ be a positive integer satisfying $\Phi_m(x)|f(x)$, and let $g(x)=\frac{f(x)}{\Phi_m(x)}$. Then
\[\|g(x)\|\leq e^{n^\varepsilon}\| f(x)\|.
\]
\end{lemma}

\begin{lemma}\label{l3} Let $\varepsilon>0$ and let $K$ be a real number such that $K\geq K_0(\varepsilon)$. Set $C=10^5\log K$. Then
\[\sum_{r=1}^K C^{\omega(r)}\leq K^{1+\epsilon}.
\]
\end{lemma}

The following lemma is a generalization of Lemma 4 in \cite{e}.

\begin{lemma}\label{l4} Let $f(x)\in\mathbb{Z}[x]$ be such that $f(1)\neq 0$ and $\|f(x)\| \leq n^c$, where $n=\deg (f(x))$ and $c$ is any constant such that $c<100\log2-1$. Suppose that $\Phi_d(x)^V|f(x)$, for some $d\in\mathbb{Z}_{>0}$. Then there exists an integer $n_0>1$ such that $n\ge n_0$ implies $V\leq (100\log n)^{\omega(d)}$.
\end{lemma}

\begin{proof} Let $p$ be a prime number dividing $d$. We write $d=p^rd_1$, where $r\geq1$ and $(p,d_1)=1$. If $U$ is any integer satisfying $0\leq U\leq V$, we obtain
\[\Phi_d(x)^{V-U}|f^{(U)}(x),\]
where $f^{(U)}(x)$ denotes the $U$-th derivative of $f(x)$. The products $\prod_{\xi\in \mu_{d_1}}\Phi_d(\xi)^{V-U}$ and $\prod_{\xi\in \mu_{d_1}}f^{(U)}(\xi)$ belong to the ring of integers of the number field $\mathbb{Q}(\xi_{d_1})$, where $\xi_{d_1}$ denotes the $d_1$-th primitive root of unity. On the other hand, they are fixed by all the automorphism in $\textrm{Gal}({\mathbb{Q}(\xi_{d_1})/\mathbb{Q}})$, so they belong to $\mathbb{Q}$. We obtain that $\prod_{\xi\in \mu_{d_1}}\Phi_d(\xi)^{V-U}$, $\prod_{\xi\in \mu_{d_1}}f^{(U)}(\xi)\in\mathbb{Z}$, and
\[\left(\prod_{\xi\in \mu_{d_1}}\Phi_d(\xi)\right)^{V-U}\bigg | \prod_{\xi\in \mu_{d_1}}f^{(U)}(\xi).
\]
We denote $N=\prod_{\xi\in \mu_{d_1}}f^{(U)}(\xi)$. We have that $|f^{U}(\xi)|\le\|f^{(U)}(x)\|\leq n^U\|f(x)\|\leq n^{U+c}$, so we obtain
\begin{eqnarray}\label{0} |N|\le (n^{U+c})^{\phi(d_1)}.
\end{eqnarray}
On the other hand, 
\[\prod_{\xi\in \mu_{d_1}}\Phi_d(\xi)=\prod_{\xi_1,\xi_2\in \mu_{d_1}}\prod_{\eta\in \mu_{p^r}}(\xi_1-\xi_2\eta).
\]
The $\xi_1=\xi_2$ part of the right-hand side is
\[\prod_{\xi\in \mu_{d_1}}\xi\prod_{\eta\in \mu_{p^r}}(1-\eta)=\prod_{\xi\in \mu_{d_1}}\Phi_{p^r}(1)\xi=p^{\phi(d_1)}\prod_{\xi\in \mu_{d_1}}\xi.
\]
Hence,
\[(p^{\phi(d_1)})^{V-U}|N.
\]
Hence, if $N\ne 0$, we have $(p^{V-U})^{\phi(d_1)}\le |N|$. Using (\ref{0}), we obtain
\[p^{V-U}\le n^{U+c}.\]

Now, let us assume that $V\ge 100 \log n$. Let $U$ be such that $0\le U\leq \frac{V}{100\log n}$. We have
\begin{align*}
2^{V(1-\frac{1}{100\log n})}&\le 2^{V-U}\le p^{V-U} \le n^{U+c}\le n^{\frac{V}{100 \log n}}n^{c}\\
          & \hspace{5.5cm} =e^{\frac{V}{100}}n^{c}\le e^{(c+1)\frac{V}{100}},\\ 
\end{align*}
so
\[e^{V(1-\frac{1}{100\log n})\log2}\le e^{(c+1)\frac{V}{100}}
\]
has to be satisfied. This is equivalent to
\[\frac{\log 2}{\log n}+c\ge 100\log2-1.
\]
But $c< 100\log2-1$, so there exists $n_0$ such that if $n\ge n_0$ the last inequality doesn't hold, hence $N=0$. This implies that if $n\ge n_0$ and $V\ge 100 \log n$, we have
$\prod_{\xi\in \mu_{d_1}}f^{(U)}(\xi)=0$, for all $0\le U \le \frac{V}{100\log n}$. Hence if $n\ge n_0$ and $V\ge 100 \log n$, we have 
\[\Phi_{d_1}(x)^{U+1}|f(x),\]
for all $0\le U \le \frac{V}{100\log n}$. 

Let $d=p_1^{r_1}\cdots p_k^{r_k}$, where $p_i$ are distinct prime numbers and $r_i>0$ for all $i$. We repeat the step above $k=\omega(d)$ times. We obtain that if $V>(100 \log n)^{\omega(d)}$, then $f(1)=0$, a contradiction. 
\end{proof}

We present the proof of the generalization of the main theorem in \cite{e}. Throughout the proof we will follow Biro's argument.

\begin{theorem}\label{t1}
For every $\varepsilon>0$ there exists an integer $n_0$ with the following property: Let $q(x)\in\mathbb{Z}[x]$. Assume that there is a polynomial $f(x)\in\mathbb{Z}[x]$ such that $f(1)\ne 0$ and there exists a positive integer $m$ such that $q(x)$ divides $(x^m-1)f(x)$. Let us denote by $n$ the degree of polynomial $(x^m-1)f(x)$ and assume that $n\geq n_0$. If $\|f(x)\|\leq n^{c}$, where $c$ is any constant such that $c<100\log2-2$ and there exists a positive integer $k$ such that $q(x)$ divides $l(x^k-1)$, for some integer $l$, then for the smallest such $k$ we have
\[\log k\leq n^{1/3+\varepsilon}
\]

\end{theorem}

\begin{proof}
Let $q(x)\mid l(x^{k'}-1)$, for some $k',l\in \mathbb{Z}$. Then
\[q(x)=l'\prod_{d\in D} \Phi_d(x), 
\]
where $l'\in\mathbb{Z}$ such that $l'|l$ and $D$ is a set of some divisors of $k'$. If we set $k=\textrm{lcm}\{d\mid d\in D\}$, we will obtain
\[q(x)\mid l(x^k-1).
\]
Our aim is to give an upper bound for $k$.

On the other hand,
\begin{eqnarray}\label{1}\deg(q(x))=\sum_{d\in D}\phi(d)\leq n.
\end{eqnarray}
Let $M$ be an integer and $L$ any real number satisfying
\[n^{\frac{1}{3}}<M<\frac{1}{2}n^{\frac{1}{2}}, \textrm{  } L>2n^{\frac{1}{2}}.
\]

In order to estimate $k$, we will estimate the products of prime factors of $k$.

\emph{Case 1.} Let $p$ be a prime such that $p^r\parallel k$, for some $r\ge1$ and $p^r\ge L$. Then $p^r|d$, for some $d\in D$. Moreover, every such $p^r$ divides a different $d\in D$, since otherwise we would have
\[\phi(d)\ge\phi(p_1^{r_1})\phi(p_2^{r_2})\ge(\frac{1}{2}p_1^{r_1})(\frac{1}{2}p_2^{r_2})\ge \frac{L^2}{4} > n,
\]
which would contradict (\ref{1}). We obtain
\[\frac{1}{2}\sum_{\substack{p^r\parallel k\\p^r \ge L}} L\le \frac{1}{2}\sum_{\substack{p^r\parallel k\\p^r\ge L}}p^r\le \sum_{\substack{p^r\parallel k\\p^r \ge L}}\phi(p^r)\le\sum_{d\in D}\phi(d)\le n.
\]
Moreover, for every such $p^r$ we have $p^r\le 2n$, and the number of such prime powers is at most $\frac{2n}{L}$, so
\begin{eqnarray}\label{2}\prod_{\substack{p^r\parallel k\\p^r\ge L}}p^r\le (2n)^{\frac{2n}{L}}.
\end{eqnarray}

\emph{Case 2.} Let $p$ be a prime such that $p^r\parallel k$, for some $r\ge1$ and $p^r\le M$. 
\begin{eqnarray}\label{3}\prod_{\substack{p^r\parallel k\\p^r\le M}}p^r\le\prod_{p^r\leq M}p^r\le e^{c_1M}.
\end{eqnarray}

\emph{Case 3.} Let $p$ be a prime such that $p^r\parallel k$, for some $r\ge2$ and $M<p^r<L$. Similarly as in the previous case, we obtain
\begin{eqnarray}\label{4}\prod_{\substack{p^r\parallel k\\M<p^r<L,r\ge 2}}p^r\le\prod_{p^r<L,r\ge 2}p^r\le e^{c_2\sqrt{L}}.
\end{eqnarray}

\emph{Case 4.} Let $p$ be a prime such that $p\parallel k$ and $M<p<L$. We need to estimate
\[\prod_{\substack{p\parallel k\\M<p<L}}p.
\]
Every such $p$ divides a $d\in D$, but a given $d\in D$ is divisible by at most two such primes, since otherwise we would have
\[\phi(d)\ge \phi(p_1)\phi(p_2)\phi(p_3)\ge (p_1-1)(p_2-1)(p_3-1)\ge M^3> n,
\]
which would contradict (\ref{1}).
Similarly, if $d\in D$ is divisible by two such primes, we have $\phi(d)\ge M^2$, so the number of primes $p\parallel k$ with $M<p<L$ for which there is another such prime $p'$ and a $d\in D$ with $p,p'|d$, is at most $\frac{2n}{M^2}$. Whence,
\[\prod_{\substack{p\parallel k\\M<p<L}}p\le L^{\frac{2n}{M^2}}\prod_{p\in \mathcal{P}}p,
\]
where $\mathcal{P}\subseteq\{p\parallel k\mid M<p<L\}$ is such that each $d\in D$ is divisible by at most one $p\in \mathcal{P}$.

We obtain that for every $p \in\mathcal{P}$ there is a $d_p\in D$ such that $p|d_p$ and if $p_1,p_2\in\mathcal{P}$ are two distict primes then $d_{p_1}\ne d_{p_2}$. By (\ref{1}), $\sum_{p\in \mathcal{P}}\phi(d_p)\leq n$.

On the other hand, if $p\in\mathcal{P}$, then $p\parallel k$, whence $p\parallel d_p$ and
\[\phi(d_p)=(p-1)\phi\left(\frac{d_p}{p}\right)\ge c_3(\varepsilon)M\left(\frac{d_p}{p}\right)^{1-\varepsilon}.
\]

Let $K$ be a real number. Then, if $p\in\mathcal{P}$ and $\frac{d_p}{p}\ge K$ we will have $\phi(d_p)\geq c_3(\epsilon)MK^{1-\epsilon}$, so the number of such primes is at most $\frac{n}{c_3(\varepsilon)MK^{1-\varepsilon}}$. Hence
\begin{eqnarray}\label{5}\prod_{p\in \mathcal{P}}p\leq L^{c_4(\epsilon)\frac{n}{MK^{1-\epsilon}}}\prod_{p\in \mathcal{P'}}p,
\end{eqnarray}
where $\mathcal{P'}=\{p\in\mathcal{P}\mid1\leq\frac{d_p}{p}<K\}$.

We partition $\mathcal{P'}$ into subsets
\[\mathcal{P'}=\bigcup_{1\leq r<K}\mathcal{P}_r,
\]
where $\mathcal{P}_r=\{p\in\mathcal{P'}\mid \frac{d_p}{p}=r\}$.

We fix an $r$ such that $1\leq r<K$. Let $V_r\geq0$ be the largest integer with the property
\[\Phi_r(x)^{V_r}\mid (x^m-1)f(x)
\]
and let
\[g(x)=\frac{(x^m-1)f(x)}{\Phi_r(x)^{V_r}}.
\]
By Lemma \ref{l2}, we have
\[\|g(x)\|\le e^{V_rn^\varepsilon}\|(x^m-1)f(x)\|\le2n^{c}e^{V_rn^\varepsilon}.
\]
Let 
\[\nu=\prod_{\xi\in \mu_r}g(\xi).\]
\noindent Then $\nu\in\mathbb{Z}$, $\nu\neq0$ and $|\nu|\leq\|g\|^{\phi(r)}$.

On the other hand, $\prod_{p\in \mathcal{P}_r}\Phi_{pr}(x) |q(x)$ and $q(x)|(x^m-1)f(x)$, so \\$\prod_{p\in \mathcal{P}_r}\Phi_{pr}(x) |(x^m-1)f(x)$. Moreover, $(\Phi_{r}(x),\prod_{p\in \mathcal{P}_r}\Phi_{pr}(x))=1$ and we obtain
\[\prod_{p\in \mathcal{P}_r}\Phi_{pr}(x)\mid g(x).
\]
Hence,
\[\prod_{p\in \mathcal{P}_r}(\prod_{\xi\in \mu_r}\Phi_{pr}(\xi))\mid\nu.
\]
For every $p\in \mathcal{P}_r$, we have $(p,r)=1$, so
\[\prod_{\xi\in \mu_r}\Phi_{pr}(\xi)=\prod_{\xi_1,\xi_2\in \mu_r}\prod_{\eta\in \mu_p}(\xi_1-\xi_2\eta).
\]
Setting $\xi_1=\xi_2$, we obtain that the right-hand side is divisible by $p^{\phi(r)}$. It follows that
\[\prod_{p\in \mathcal{P}_r}p^{\phi(r)}\mid \nu,
\]
and 
\[\prod_{p\in \mathcal{P}_r}p^{\phi(r)}\le |\nu|\le\|g(x)\|^{\phi(r)}.
\]
Consequently,
\[\prod_{p\in \mathcal{P}_r}p\leq\|g(x)\|\leq 2n^{c}e^{V_rn^\varepsilon}.
\]
This implies that if $n$ is sufficiently large, we have
\[|\mathcal{P}_r|\leq c_5(V_r+1)n^\epsilon,
\]
so
\begin{eqnarray}\label{6}\prod_{p\in \mathcal{P'}}p\leq L^{c_5{n^\epsilon}\sum_{1\leq r<K}(V_r+1)}.
\end{eqnarray}

Combining (\ref{2}), (\ref{3}), (\ref{4}), (\ref{5}) and (\ref{6}), we obtain that if $n$ is sufficiently large, then

\begin{align*}
\log k&\le\frac{2n}{L}\log(2n)+c_1M+c_2\sqrt{L}+c_4(\varepsilon)\frac{n}{MK^{1-\varepsilon}}\log L\\
          & \hspace{6cm} +c_5{n^\epsilon}\sum_{1\leq r<K}(V_r+1)\log L\\ 
          &\le \frac{2n}{L}\log(2n)+c_1M+c_2\sqrt{L}\\
          & \hspace{2cm} +c_6(\epsilon)(\log L)(Kn)^\epsilon (\frac{n}{M^2}+\frac{n}{MK}+\sum_{1\leq r<K}(V_r+1))
\end{align*}

The $L$ part is optimized by taking $L=n^{\frac{2}{3}}$. Next, we estimate $V_r+1$, using Lemma \ref{l4}. Let
\[h(x)=\frac{(x^m-1)f(x)}{x-1}.
\]
Then $h(x)\in\mathbb{Z}[x]$ is a nonzero polynomial such that $h(1)\ne0$. Also, we have $\deg(h(x))\leq n$ and $\|h(x)\|\le m\|f(x)\|\le n^{c+1}$, where $c< 100\log2-2$.
Now, if $r>1$, we have $\Phi_r(x)^{V_r}\mid h(x)$ and by Lemma \ref{l4}, $V_r\leq (100\log n)^{\omega(r)}$, for sufficiently large $n$. We obtain
\[V_r+1\leq (200\log n)^{\omega(r)}, \textrm{ for } r>1 \textrm{ and } V_1+1=2.
\]

Assuming a weak estimate $K\ge n^{\frac{1}{100}}$ and using Lemma \ref{l3}, we finally have
\[\log k\leq c_7(\epsilon)(Kn)^{2\epsilon} (n^{\frac{1}{3}}+M+\frac{n}{M^2}+\frac{n}{MK}+K).
\]
This is nearly optimized in $K$ by $K=(\frac{n}{M})^{\frac{1}{2}}$, and the remaining expression is nearly optimized in $M$ with $M=n^{\frac{1}{3}}$. We fix the parameters
\[K=n^{\frac{1}{3}},\textrm{   }, M=\lfloor2n^{\frac{1}{3}}\rfloor,\textrm{   }, L=n^{\frac{2}{3}}.
\]
and obtain
\[\log k\leq n^{\frac{1}{3}+10\epsilon},
\]
for sufficiently large n.
\end{proof}

\section{Multisets and periodicity}

In this section we prove Theorem 4 and Theorem 5. Let $A$ be a finite multiset of integers and let $B$ be a multiset such that $A\oplus_t B=\mathbb{Z}$. We denote $\alpha=\min(A)$. We define a multiset $A'=A-\alpha=\{a-\alpha\mid a\in A\}$, with $w_{A'}(a-\alpha)=w_A(a)$ for all $a\in A$. If $n\in\mathbb{Z}$, we have

\begin{align*}R_{A',B}(n)&=\sum_{\substack{n=a'+b\\a'=a-\alpha\in A', b\in B}}w_{A'}(a')w_B(b)\\
&=\sum_{\substack{n+\alpha=a+b\\a\in A, b\in B}}w_A(a)w_B(b)=R_{A,B}(n+\alpha).
\end{align*}

Thus, $A'\oplus_tB=\mathbb{Z}$. Hence, we may assume without loss of generality that $\min(A)=0$ and $\max(A)=d$, where $d=\textrm{diam}(A)$.

Let $|A|=1$. Then if $A=\{a\}$ and $A\oplus_tB=\mathbb{Z}$, we obtain that $B=\mathbb{Z}$ and $w_B(n)=\frac{t}{w_A(a)}$, for all $n\in\mathbb{Z}$. Hence, the multiset $B$ is periodic with period $k=1$ and the Theorem is immediately true. Thus, we may assume that $|A|>1$ and $d\ge 1$.

We have
\begin{align*}t=R_{A,B}(n)&=\sum_{a\in A}w_A(a)w_B(n-a)\\
                                             &=\sum_{a\in A\setminus\{0\}}w_A(a)w_B(n-a)+w_A(0)w_B(n),
\end{align*}
for all $n\in\mathbb{Z}$.
Thus,
\begin{align}\label{rec1} w_A(0)w_B(n)=t-\sum_{a\in A\setminus\{0\}}w_A(a)w_B(n-a).
\end{align}
We have that $n-d\le n-a\le n-1$, for all $a\in A\setminus\{0\}$. 
Moreover,
\begin{align}\label{rec2} w_A(d)w_B(n-d)=t-\sum_{a\in A\setminus\{d\}}w_A(a)w_B(n-a),
\end{align}
and $n-d+1\le n-a\le n$, for all $a\in A\setminus\{d\}$.
Hence, if we know the value of $w_B$ for any $d$ consecutive integers, using (\ref{rec1}) and (\ref{rec2}) we can compute $w_B(n)$ for all integers $n$. 

We consider the $d$-tuple $(w_B(i), w_B(i+1),\hdots, w_B(i+d-1))$, for some $i\in\mathbb{Z}$. Since 
\[t=R_{A,B}(n+a)=\sum_{a\in A}w_A(a)w_B(n-a)\ge w_A(0)w_B(n),\]
we obtain that $w_B(n)\le t$, for all $n\in\mathbb{Z}$. Hence,
\[(w_B(i), w_B(i+1),\hdots, w_B(i+d-1))\in\left\{0,1,\hdots,t\right\}^d.
\]
By pigeonhole principle, there exist integers $0\le i<j\le (t+1)^d$ such that 
{\footnotesize
\[(w_B(i), w_B(i+1),\hdots, w_B(i+d-1))=(w_B(j), w_B(j+1),\hdots, w_B(j+d-1)).
\]}
Let $k=j-i$. Then $1\le k\le (t+1)^d$ and $w_B(n)=w_B(n+k)$, for $n=i,\hdots,i+d-1$. By (\ref{rec1}) and (\ref{rec2}), we obtain $w_B(n)=w_B(n+k)$, for all $n\in\mathbb{Z}$. This proves the Theorem 4.

\emph{Proof of the Theorem 5.} As before, we may assume without loss of generality that $\min(A)=0$ and $\max(A)=d$, where $d=\textrm{diam}(A)$. Moreover, let $\beta\in B$. We define the mutiset $B'=B-\beta=\{b-\beta\mid b\in B\}$ with $w_{B'}(b-\beta)=w_B(b)$, for all $b\in B$. Then, $0\in B'$ and $B'$ is eventually periodic with eventual period $k$ if and only if $B$ is eventually periodic with eventual period $k$. If $n\in\mathbb{Z}$, we have

\begin{align*}R_{A,B'}(n)&=\sum_{\substack{n=a+b'\\a\in A, b'=b-\beta\in B'}}w_{A}(a)w_B'(b')\\
&=\sum_{\substack{n+\beta=a+b\\a\in A, b\in B}}w_A(a)w_B(b)=R_{A,B}(n+\beta).
\end{align*}

It follows that the representation function $R_{A,B'}$ is eventually periodic with eventual period $m$ if and only if $R_{A,B}$ is eventually periodic with eventual period $m$. Thus, we may assume without loss of generality that $0\in B$.

Let
\[B^+=\{b\in B\mid b\ge 0\}.
\]
Note that if $b\in B^+$, we have that $a+b \ge 0$, for all $a\in A$. On the other hand, if $b\in B\setminus B^+$, we have $a+b < \textrm{diam}(A)$, for all $a\in A$. Hence, $R_{A,B^+}(n)=R_{A,B}(n)$, for all $n\ge \textrm{diam}(A)$ and $R_{A,B^+}$ is eventually periodic with eventual period $m$. 

We consider the generating functions
\[\lambda(x)=\sum_{a\in A}w_A(a)x^{a}=\sum_{n=0}^\infty w_A(n)x^{n}     
\]
and
\[\gamma(x)=\sum_{b\in B^+} w_B(b)x^b=\sum_{n=0}^\infty w_B(n)x^{n}.
\]
Let $n_0\in\mathbb{Z}_{>0}$ be such that $R_{A,B^+}(n+m)=R_{A,B^+}(n)$, for all integers $n\ge n_0$. Let $i_0=\lfloor \frac{n_0}{m}\rfloor+1$. Define $r_j=R_{A,B^+}(mi_0+j)$ for $j=0,1,\hdots,m-1$. We obtain
\begin{align*}
\lambda(x)\gamma(x)&=\sum_{n=0}^{\infty}R_{A,B^+}(n)x^n\\
&=\sum_{n=0}^{m i_0-1}R_{A,B^+}(n)x^n+\sum_{n=m i_0}^{\infty}R_{A,B^+}(n)x^n\\
&=\sum_{n=0}^{m i_0-1}R_{A,B^+}(n)x^n+\sum_{j=0}^{m-1}\sum_{i= i_0}^{\infty}R_{A,B^+}(m i+j)x^{m i+j}
\end{align*}
\begin{align}
&=\sum_{n=0}^{m i_0-1}R_{A,B^+}(n)x^{n}+\sum_{j=0}^{m-1}\sum_{i= i_0}^{\infty}r_jx^{m i+j}\nonumber\\
&=\sum_{n=0}^{m i_0-1}R_{A,B^+}(n)x^n-\frac{1}{x^{m}-1}\sum_{j=0}^{m-1}r_jx^{m i_0+j}.
\end{align}
 
We define the polynomials
\[p_1(x)=\sum_{n=0}^{m i_0-1}R_{A,B^+}(n)x^n
\]
and
\[p_2(x)=\sum_{j=0}^{m-1}r_jx^{m i_0+j}.
\]
Then, using $(10)$, we obtain 
\begin{align*}\lambda(x)\gamma(x)&=p_1(x)-\frac{1}{x^{m}-1}p_2(x)\\
&=\frac{(x^{m}-1)p_1(x)-p_2(x)}{x^{m}-1},
\end{align*}
and
\[\gamma(x)=\frac{(x^m-1)p_1(x)-p_2(x)}{(x^m-1)\lambda(x)}=\frac{p(x)}{q(x)},
\]
where $p(x)$ and $q(x)$ are relatively prime polynomials in $\mathbb{Z}[x]$ and $q(x)|(x^m-1)\lambda(x)$. The multiset $B^+$ is eventually periodic, hence there exists $n_1\in\mathbb{Z}_{>0}$ and $s\in\mathbb{Z}_{>0}$ such that $w_B(n+s)=w_B(n)$, for all $n\ge n_1$. Hence,
\begin{align*}(x^s-1)\gamma(x)&=(x^s-1)\sum_{n=0}^\infty w_B(n)x^{n}\\
&=\sum_{n=0}^\infty w_B(n)x^{n+s}-\sum_{n=0}^\infty w_B(n)x^n\\
&=\sum_{n=0}^{n_1-1}w_B(n)x^{n+s}+\sum_{n=n_1}^\infty w_B(n)x^{n+s}-\sum_{n=0}^\infty w_B(n)x^n\\
&=\sum_{n=0}^{n_1-1}w_B(n)x^{n+s}+\sum_{n=n_1}^\infty w_B(n+s)x^{n+s}-\sum_{n=0}^\infty w_B(n)x^n
\end{align*}
\begin{align*}
&=\sum_{n=0}^{n_1-1}w_B(n)x^{n+s}+\sum_{n=n_1+s}^\infty w_B(n)x^{n}-\sum_{n=0}^\infty w_B(n)x^n\\
&=\sum_{n=0}^{n_1-1}w_B(n)x^{n+s}-\sum_{n=0}^{n_1+s} w_B(n)x^n,
\end{align*}

\noindent and $(x^s-1)\gamma(x)$ is a polynomial. Then
\[(x^{s}-1)\gamma(x)=\frac{(x^{s}-1)p(x)}{q(x)}
\]
is a polynomial, whence $q(x)|(x^{s}-1)p(x)$. Since gcd$(p(x),q(x))=1$, we conclude that there exists an integer $l$ such that $q(x)| l(x^{s}-1)$. We have $\|\lambda(x)\|=\sum_{a\in A}w_A(a)$ and $\lambda(1)=\sum_{a\in A}w_A(a)\ne 0$. The conditions of Theorem \ref{t1} are fulfilled. Then there exists a positive integer $k$ such that
\[q(x)|l'(x^{k}-1) \textrm{ for some integer $l'$ and } \log (k)\leq n^{\frac{1}{3}+\epsilon}.
\]
It remains to prove that $k$ is an eventual period of $B^+$. We have that $l'(x^{k}-1)\gamma(x)\in\mathbb{Z}[x]$. This implies that $(x^{k}-1)\gamma(x)$ is a polynomial, so
\begin{align*}(x^k-1)\gamma(x)&=(x^k-1)\sum_{n=0}^\infty w_B(n)x^{n}\\
&=\sum_{n=0}^\infty w_B(n)x^{n+k}-\sum_{n=0}^\infty w_B(n)x^n\\
&=\sum_{n=0}^\infty w_B(n)x^{n+k}-\sum_{n=n_1+k}^\infty w_B(n)x^n-\sum_{n=0}^{n_1+k-1} w_B(n)x^n\\
&=\sum_{n=0}^\infty w_B(n)x^{n+k}-\sum_{n=n_1}^\infty w_B(n+k)x^{n+k}-\sum_{n=0}^{n_1+k-1} w_B(n)x^n\\
&=\sum_{n=n_1}^\infty (w_B(n)-w_B(n+k))x^{n+k}\\
&\hspace{3cm}+\sum_{n=0}^{n_1-1} w_B(n)x^{n+k}-\sum_{n=0}^{n_1+k-1} w_B(n)x^n
\end{align*}

\noindent is a polynomial. Thus, there exist $n_2\in\mathbb{Z}_{>0}$ such that $w_B(n)=w_B(n+k)$, for all $n\ge n_2$. This implies that $k$ is an eventual period of $B^+$.

\end{document}